\documentclass[11pt,reqno]{amsart}
 \usepackage{amssymb,amsfonts,amscd,amsbsy, color, amsmath}
\usepackage[mathscr]{eucal}
\usepackage{url}
\usepackage{graphicx}
\usepackage{mathtools}

\newtheorem{theorem}{Theorem}[section]

\newtheorem{proposition}[theorem]{Proposition}

\theoremstyle{definition}
\newtheorem{remark}[theorem]{Remark}
\numberwithin{equation}{section}

\makeatletter
\def\imod#1{\allowbreak\mkern5mu({\operator@font mod}\,\,#1)}
\makeatother
\allowdisplaybreaks

\begin{document}

\title[Generalized Fishburn numbers and torus knots]{Generalized Fishburn numbers and torus knots}

\author[C. Bijaoui]{Colin Bijaoui}

\author[H. U. Boden]{Hans U. Boden}

\author[B. Myers]{Beckham Myers}

\author[R. Osburn]{Robert Osburn}

\author[W. Rushworth]{William Rushworth}

\author[A. Tronsgard]{Aaron Tronsgard}

\author[S. Zhou]{Shaoyang Zhou}

\address{Department of Mathematics $\&$ Statistics, McMaster University, Hamilton Hall, 1280 Main Street West, Hamilton, Ontario, Canada L8S 4K1}

\email{bijaoucg@mcmaster.ca, boden@mcmaster.ca, will.rushworth@math.mcmaster.ca}

\address{Department of Mathematics, Harvard University, Science Center Room 325, 1 Oxford Street, Cambridge, MA 02138, USA}

\email{bmyers@college.harvard.edu}

\address{School of Mathematics and Statistics, University College Dublin, Belfield, Dublin 4, Ireland}

\email{robert.osburn@ucd.ie}

\address{Department of Mathematics, University of Toronto, Toronto, Ontario, M5S 2E4, Canada}

\email{tronsgar@math.toronto.edu}

\address{Department of Mathematics, Vanderbilt University, Nashville, TN 37240, USA}

\email{shaoyang.zhou@vanderbilt.edu}


\date{\today}

\begin{abstract}
Andrews and Sellers recently initiated the study of arithmetic properties of Fishburn numbers. In this paper we prove prime power congruences for generalized Fishburn numbers. These numbers are the coefficients in the $1-q$ expansion of the Kontsevich-Zagier series $\mathscr{F}_{t}(q)$ for the torus knots $T(3,2^t)$, $t \geq 2$. The proof uses a strong divisibility result of Ahlgren, Kim and Lovejoy and a new ``strange identity" for $\mathscr{F}_{t}(q)$. 
\end{abstract}

\maketitle

\section{Introduction}

The Fishburn numbers $\xi(n)$ are the coefficients in the formal power series expansion

\begin{equation} \label{fish}
F(1-q) =: \sum_{n \geq 0} \xi(n) q^n = 1 + q + 2q^2 + 5q^3 + 15q^4 + 53q^5 + \cdots
\end{equation}

\noindent where $F(q) := \sum_{n \geq 0} (q)_n$ is the Kontsevich-Zagier ``strange" series \cite{z1} and

\begin{equation*}
(a_1, a_2, \dotsc, a_j)_n = (a_1, a_2, \dotsc, a_j ; q)_n := \prod_{k=1}^{n}(1-a_1 q^{k-1})(1 - a_2 q^{k-1}) \cdots (1-a_j q^{k-1})
\end{equation*} 

\noindent is the standard $q$-hypergeometric notation, valid for $n \in \mathbb{N} \cup \{ \infty \}$. 
Here, the moniker ``strange" is used as $F(q)$ does not converge on any open subset of $\mathbb{C}$, but is well-defined when $q$ is a root of unity (where it is finite) and when $q$ is replaced by $1-q$ as in (\ref{fish}). The Fishburn numbers are of interest for their numerous combinatorial variants (see A022493 in \cite{oeis}), asymptotics \cite{sto, z1} and arithmetic properties \cite{ak, as, garvan, gkr, straub}. In their marvelous paper, Andrews and Sellers \cite{as} proved congruences for $\xi(n)$ modulo primes which were then extended to prime powers \cite{ak, straub}. For example, we have

\begin{equation} \label{mod5}
\xi(5^r m - 1) \equiv \xi(5^r m - 2) \equiv 0 \pmod{5^r}, 
\end{equation}
\begin{equation} \label{mod7}
\xi(7^r m - 1) \equiv 0 \pmod{7^r}
\end{equation}
\noindent and
\begin{equation} \label{mod11}
\xi(11^r m -1) \equiv \xi(11^r m -2) \equiv \xi(11^r m - 3) \equiv 0 \pmod{11^r}
\end{equation}

\noindent for all natural numbers $r$ and $m$. Our interest in this paper lies in the knot theoretic interpretation of $F(q)$ as it leads to a natural generalization of the coefficients $\xi(n)$.

Let $K$ be a knot and $J_N(K;q)$ be the usual colored Jones polynomial, normalized to be $1$ for the unknot. As a knot invariant, the colored Jones polynomial plays the lead role in many open problems in quantum topology. The sequence $\{J_N(K;q)\}_{N \in \mathbb{N}}$ appears to encode many subtle geometric and topological properties of the knot $K$ at a remarkably deep level. For example, the Volume Conjecture  \cite{kashaev, hm, murakami-murakami} relates the value at $\zeta_N = e^{2\pi i/N}$ of the $N$th colored Jones polynomial (or, equivalently, the $N$th Kashaev invariant) of a knot to its hyperbolic volume; the Strong Slope Conjecture \cite{garoufalidis-slope, kal-tran} posits that the maximal and minimal degrees of $J_N(K;q)$ in $q$ contain information about essential surfaces in knot exteriors; the AJ Conjecture \cite{garoufalidis-AJ} connects the recurrence relation for $J_N(K;q)$ to the $A$-polynomial of the knot \cite{ccgls}, a plane curve describing the $SL(2,{\mathbb C})$ character variety of the knot complement. Explicit formulas for $J_{N}(K;q)$ in terms of $q$-hypergeometric series are also of importance and have been proven for various families of knots \cite{habiro, hikami1, hikami2, thang, loCJP, loCJP2, masbaum, takata, walsh}. For example, if $T(3,2)=T(2,3)$ is the right-handed trefoil knot as in Figure \ref{fig:trefoil}, then \cite{habiro, thang}

\begin{equation} \label{t32}
J_N(T(3,2); q) = q^{1-N} \sum_{n \geq 0} q^{-nN} (q^{1-N})_n.
\end{equation}

\begin{figure}[ht]
\includegraphics[width=.2\linewidth]{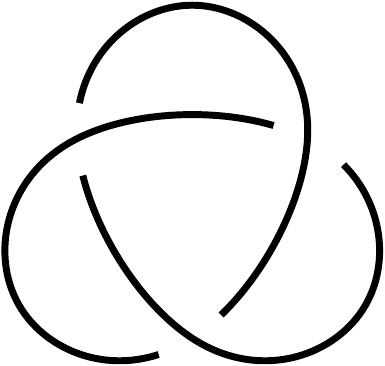}
\caption{$T(3,2)$}
\label{fig:trefoil}
\end{figure}

\noindent Observe that the Kontsevich-Zagier series $F(q)$ matches the colored Jones polynomial for $T(3,2)$ at roots of unity, that is, for $q=\zeta_N$ we have

\begin{equation} \label{Fcjpmatch}
\zeta_N F(\zeta_N) = J_N(T(3,2); \zeta_N).
\end{equation}

Consider the family of torus knots $T(3, 2^t)$ for $t \geq 2$ as in Figures \ref{fig:t34} and \ref{fig:t32t}. 

\begin{figure}[ht]
\centering
\begin{minipage}{.5\textwidth}
  \centering
  \includegraphics[width=.4\linewidth]{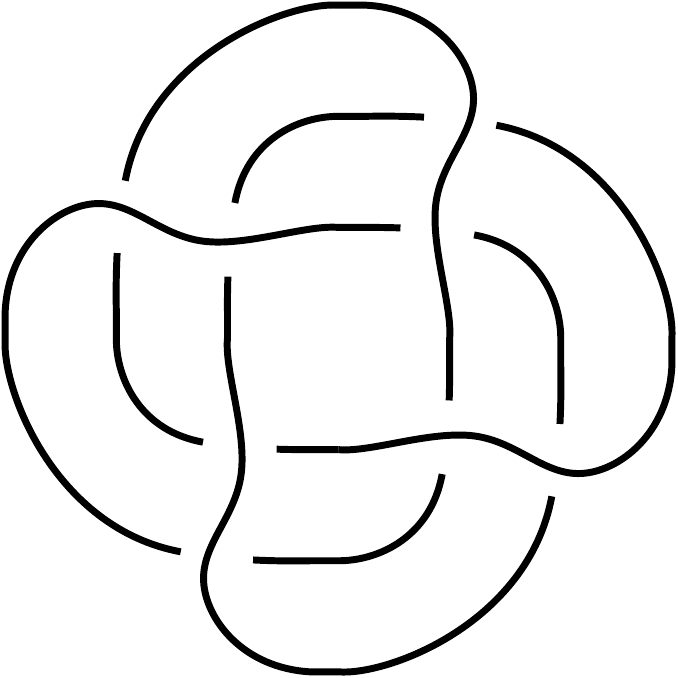}
  \caption{$T(3,4)$}
  \label{fig:t34}
\end{minipage}%
\begin{minipage}{.5\textwidth}
  \centering
  \includegraphics[width=.6\linewidth]{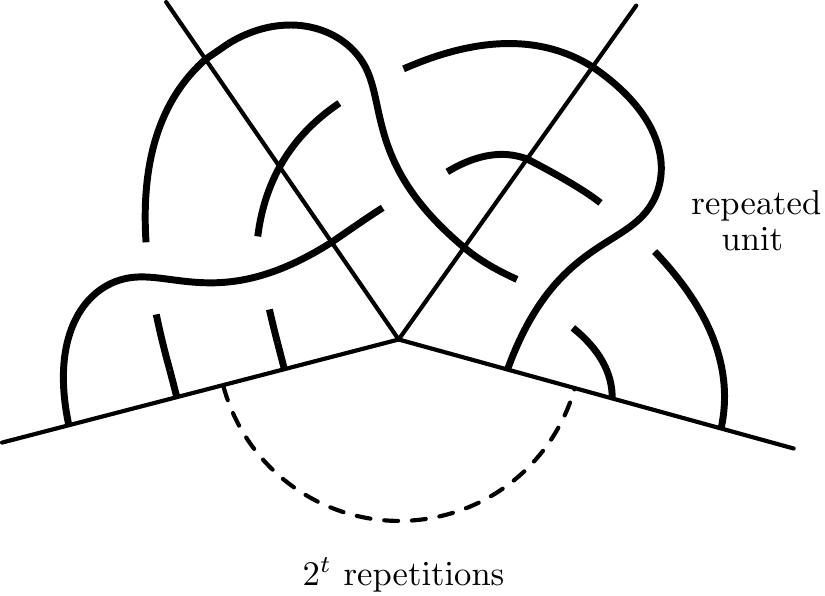}
  \caption{$T(3,2^t)$}
  \label{fig:t32t}
\end{minipage}
\end{figure}

\noindent In this case, a $q$-hypergeometric expression for the colored Jones polynomial has been computed, namely (see page 41, Th{\'e}or{\`e}me 3.2 in \cite{konan})

\begin{align} \label{t32t}
J_N(T(3,2^t); q) & = (-1)^{h''(t)} q^{2^t - 1 - h'(t) - N} \sum_{n \geq 0} (q^{1-N})_n q^{-Nn m(t)} \nonumber \\
& \times \sum_{3 \sum_{\ell=1}^{m(t)-1} j_{\ell} \ell \equiv 1 \pmod{m(t)}} (-q^{-N})^{\sum_{\ell=1}^{m(t)-1} j_{\ell}} q^{\frac{-a(t) + \sum_{\ell=1}^{m(t)-1} j_{\ell} \ell}{m(t)} + \sum_{\ell=1}^{m(t)-1} \binom{j_{\ell}}{2}} \nonumber \\
& \quad \quad \times \sum_{k=0}^{m(t) - 1} q^{-kN} \prod_{\ell=1}^{m(t) - 1} \begin{bmatrix} n + I(\ell \leq k) \\ j_{\ell} \end{bmatrix} 
\end{align}

\noindent where 

\begin{equation*}
h''(t) = 
\begin{cases}
\frac{2^t-1}{3} &\text{if $t$ is even}, \\
\frac{2^t -2}{3} &\text{if $t$ is odd}, \\
\end{cases} \quad \quad
h'(t) =
\begin{cases}
\frac{2^t-4}{3} &\text{if $t$ is even}, \\
\frac{2^t -5}{3} &\text{if $t$ is odd}, \\
\end{cases} \quad \quad 
a(t)=
\begin{cases}
\frac{2^{t-1} +1}{3} &\text{if $t$ is even}, \\
\frac{2^t +1}{3} &\text{if $t$ is odd}, \\
\end{cases}
\end{equation*}

\noindent $m(t)=2^{t-1}$, $I(*)$ is the characteristic function and

\begin{equation*}
\begin{bmatrix} n \\ k \end{bmatrix} = \begin{bmatrix} n \\ k \end{bmatrix}_{q} := \frac{(q)_n}{(q)_{n-k} (q)_k}
\end{equation*}

\noindent is the $q$-binomial coefficient. We note that the $t=2$ case of (\ref{t32t}) recovers equation (16) in \cite{hikami1}. We now define the {\it Kontsevich-Zagier series for torus knots $T(3,2^t)$} as\footnote{\label{t1}For $t=1$, one may define the sum over the $j_{\ell}$ to be 1 in (\ref{t32t}) and (\ref{kztorus}) to recover (\ref{t32}) and $F(q)$.}

\begin{align} \label{kztorus} 
\mathscr{F}_t(q) & = (-1)^{h''(t)} q^{-h'(t)} \sum_{n \geq 0} (q)_n \sum_{3 \sum_{\ell = 1}^{m(t) - 1} j_{\ell} \ell \equiv 1 \pmod{m(t)}} q^{\frac{-a(t) + \sum_{\ell=1}^{m(t) - 1} j_{\ell} \ell}{m(t)} + \sum_{\ell=1}^{m(t)-1} \binom{j_{\ell}}{2}} \nonumber \\
& \times \sum_{k=0}^{m(t)-1} \prod_{\ell=1}^{m(t) - 1} \begin{bmatrix} n + I(\ell \leq k) \\ j_{\ell} \end{bmatrix}
\end{align}

\noindent and notice that, similar to (\ref{Fcjpmatch}), we have

\begin{equation*}
\zeta_{N}^{2^t-1} \mathscr{F}_t(\zeta_N) = J_N(T(3,2^t); \zeta_N).
\end{equation*}

\noindent As with the original Kontsevich-Zagier series, $\mathscr{F}_t(q)$ does not converge on any open subset of $\mathbb{C}$, but is well-defined when $q$ is a root of unity (where it is finite) and when $q$ is replaced by $1-q$. Thus, we may write

\begin{equation*}
\mathscr{F}_t(1-q) =: \sum_{n \geq 0} \xi_t(n) q^n.
\end{equation*}

\noindent For example, 

\begin{equation*}
\mathscr{F}_2(1-q) = 1 + 3q + 11q^2 + 50q^3 + 280q^4 + 1890q^5 + \cdots
\end{equation*}
\noindent and
\begin{equation*}
\mathscr{F}_3(1-q) = 1 + 7q + 49q^2 + 420q^3 + 4515q^4 + 59367q^5 + \cdots.
\end{equation*}

The purpose of this paper to illustrate how a recent result of Ahlgren, Kim and Lovejoy \cite{akl} combined with a new ``strange identity" for $\mathscr{F}_t(q)$ allow one to prove prime power congruences akin to (\ref{mod5})--(\ref{mod11}) for the {\it generalized Fishburn numbers} $\xi_t(n)$. For natural numbers $s$ and $t \geq 2$ and the periodic function

\begin{equation} \label{genchi}
\chi_t(n) = \chi_{3\cdot 2^{t+1}}(n) :=
\begin{cases}
1 &\text{if $n \equiv 2^{t+1}-3$, $3+2^{t+2}$ $\pmod{3\cdot2^{t+1}}$}, \\
-1 &\text{if $n \equiv 2^{t+1} + 3$, $2^{t+2} - 3$ $\pmod{3\cdot2^{t+1}}$,} \\
0 &\text{otherwise,}
\end{cases}
\end{equation} 

\noindent we define the set

\begin{equation*}
S_{t, \chi_t}(s) = \Bigl \{ 0 \leq j \leq s-1 : j \equiv \frac{n^2 - (2^{t+1} - 3)^2}{3 \cdot 2^{t+2}} \pmod{s} \,\, \text{where} \,\, \chi_{t}(n) \neq 0 \Bigr \}.
\end{equation*}
Our main result is now the following.

\begin{theorem} \label{main} If $p \geq 5$ is a prime and $j \in \{1, 2, \dotsc, p - 1 - \text{max} \,\, S_{t, \chi_t}(p) \}$, then

\begin{equation*} \label{xiprimepowers}
\xi_t(p^r m - j) \equiv 0 \pmod{p^r}
\end{equation*}

\noindent for all natural numbers $r$, $m$ and $t \geq 2$.
\end{theorem}

\noindent One can check that $S_{2, \chi_2}(5)=\{ 0, 2, 3 \}$, $S_{2, \chi_2}(17) = \{ 0, 2, 3, 4, 7, 8, 9, 11, 14 \}$, $S_{3, \chi_3}(7)=\{ 0, 2, 3, 4 \}$ and $S_{3, \chi_3}(13)=\{0,2,5,6,7,8,11\}$. Thus, by Theorem \ref{main}, we have

\begin{equation*} 
\xi_2(5^r m - 1 ) \equiv 0 \pmod{5^r}, 
\end{equation*}
\begin{equation*} 
\xi_2(17^r m- 1 ) \equiv \xi_2(17^r m - 2) \equiv 0 \pmod{17^r},
\end{equation*}
\begin{equation*} 
\xi_3(7^r m - 1) \equiv \xi_3(7^r m - 2) \equiv 0 \pmod{7^r} 
\end{equation*}
\noindent and 
\begin{equation*} \label{xi3mod13}
\xi_3(13^r m -1) \equiv 0 \pmod{{13}^r}
\end{equation*}

\noindent for all natural numbers $r$ and $m$. 

The paper is organized as follows. In Section 2, we recall the main result from \cite{akl} and then record some preliminaries, including the new ``strange identity" (see Theorem \ref{newstr}). Theorem \ref{newstr} is of independent interest for at least three reasons. First, there seem to be few such explicit instances in which the underlying identity has been proven (for example, see equation (70) in \cite{hikami2} and Theorem 2 in \cite{z1}). Second, this result implies (via Theorem \ref{strangetocong}) a strong divisibility property for the coefficients of the partial sums of the dissection of $\mathscr{F}_t(q)$. This is used in turn to prove Theorem \ref{main}. Third, it is a key component in determining the quantum modularity for $\mathscr{F}_t(q)$ \cite{ano}. In Section 3, we prove Theorem \ref{main}. Finally, in Section 4, we discuss some possibilities for future work.

\section{Preliminaries}

Our first step is to recall the setup from \cite{akl}. Let $\mathcal{F}$ be a function of the form

\begin{equation} \label{genF}
\mathcal{F}(q) = \sum_{n \geq 0} (q)_n f_n(q)
\end{equation}

\noindent where $f_n(q) \in \mathbb{Z}[q]$. For positive integers $s$ and $N$, consider the partial sum

\begin{equation*} 
\mathcal{F}(q;N) := \sum_{n=0}^{N} (q)_n f_n(q)
\end{equation*}

\noindent and its $s$-dissection 

\begin{equation*}
\mathcal{F}(q; N) = \sum_{i=0}^{s-1} q^i A_{\mathcal{F}, s}(N,i,q^s)
\end{equation*}

\noindent where $A_{\mathcal{F}, s}(N,i,q) \in \mathbb{Z}[q]$. Now, consider the partial theta functions

\begin{equation} \label{partial}
P_{a,b,\chi}^{(\nu)}(q) := \sum_{n \geq 0} n^{\nu} \chi(n) q^{\frac{n^2 - a}{b}}
\end{equation}

\noindent where $\nu \in \{0,1\}$, $a \geq 0$, $b > 0$ are integers and $\chi: \mathbb{Z} \to \mathbb{C}$ is a function satisfying the following two conditions:
\begin{equation} \label{p1}
\text{$\chi(n) \neq 0$ only if $\frac{n^2 - a}{b} \in \mathbb{Z}$}
\end{equation}

\noindent and for each root of unity $\zeta$, 
\begin{equation} \label{p2}
\text{the function $n \to \zeta^{\frac{n^2 - a}{b}} \chi(n)$ is periodic and has mean value zero.}
\end{equation}

\noindent Finally, define the set $S_{a,b,\chi}(s)$ by

\begin{equation*}
S_{a,b,\chi}(s) = \Biggl \{ 0 \leq j \leq s-1 : j \equiv \frac{n^2 - a}{b} \pmod{s}  \,\, \text{where} \,\  \chi(n) \neq 0 \Biggr \}.
\end{equation*}

\noindent We can now state the main result in \cite{akl}.

\begin{theorem} \label{strangetocong} Suppose that $\mathcal{F}$ and $P_{a,b,\chi}^{(\nu)}$ are functions as in (\ref{genF}) and (\ref{partial}) and for each root of unity $\zeta$, we have the asymptotic expansion

\begin{equation*}
\mathcal{F}(\zeta e^{-t}) \sim P_{a,b,\chi}^{(\nu)}(\zeta e^{-t})
\end{equation*}

\noindent as $t \to 0^{+}$. Suppose that $s$ and $N$ are positive integers and $i \not \in S_{a,b,\chi}(s)$. Then 

\begin{equation*}
(q)_{\lambda(N,s)} \mid A_{\mathcal{F},s}(N,i,q)
\end{equation*}

\noindent where $\lambda(N,s) := \bigl \lfloor \frac{N+1}{s} \bigr \rfloor$.

\end{theorem}

We now turn to our situation and record the following result. Throughout, we assume that $t \geq 2$. 

\begin{proposition} \label{step1} The periodic function $\chi_t$ (as defined by (\ref{genchi})) satisfies (\ref{p1}) and (\ref{p2}).
\end{proposition}

\begin{proof}
A straightforward calculation using (\ref{genchi}) confirms (\ref{p1}) with $a=(2^{t+1} - 3)^2$ and $b=3 \cdot 2^{t+2}$. Suppose $\zeta$ is a root of unity of order $M$ and define 

\begin{equation*}
\psi_{t}(n) = \zeta^{\frac{n^2 - (2^{t+1} - 3)^2}{3 \cdot 2^{t+2}}} \chi_t.
\end{equation*}

\noindent Note that $\psi_{t}$ has period $M(3\cdot 2^{t+1})$. We now claim that

\begin{equation} \label{mean}
\sum_{n=1}^{M(3\cdot 2^{t+1})} \psi_{t}(n) = 0.
\end{equation}

\noindent Suppose $M$ is even. By (\ref{genchi}), we obtain

\begin{equation} \label{chi}
\chi_t(n + 3\cdot 2^t M) = \chi_t(n)
\end{equation}

\noindent for all $n$. We also have

\begin{equation} \label{zeta}
\zeta^{\frac{(n+ 3 \cdot 2^t M)^2 - (2^{t+1} - 3)^2}{3 \cdot 2^{t+2}}} = \zeta^{\frac{M(n+ 3 \cdot 2^{t-1} M^2)}{2}} \zeta^{\frac{n^2 - (2^{t+1} - 3)^2}{3 \cdot 2^{t+2}}}.
\end{equation}

\noindent As $\psi_t$ is supported on odd integers, we can assume that $n$ is odd. Then $M(n+ 3 \cdot 2^{t-1} M^2)$ is an odd multiple of $M$ and so (\ref{zeta}) equals $-\zeta^{\frac{n^2 - (2^{t+1} - 3)^2}{3 \cdot 2^{t+2}}}$. This fact and (\ref{chi}) imply

\begin{equation*}
\sum_{n=1}^{M(3 \cdot 2^{t+1})} \psi_t(n) = \sum_{n=1}^{M(3\cdot 2^t)} (\psi_t(n) + \psi_t(n + 3 \cdot 2^t M)) = 0.
\end{equation*}

Now, suppose $M$ is odd. We break up (\ref{mean}) into four sums, one for each congruence class modulo $3\cdot 2^{t+1}$ in (\ref{genchi}). Specifically, 

\begin{align} \label{four}
\sum_{n=1}^{M(3\cdot 2^{t+1})} \psi_{t}(n) & = \sum_{m=0}^{M-1} \zeta^{3\cdot 2^t m^2 + (2^{t+2} + 3)m + 2^t + 3} - \sum_{m=0}^{M-1} \zeta^{3\cdot 2^t m^2 + (2^{t+2} - 3)m + 2^t -1} \nonumber \\
& + \sum_{m=0}^{M-1} \zeta^{3 \cdot 2^t m^2 + (2^{t+1} - 3)m} - \sum_{m=0}^{M-1} \zeta^{3 \cdot 2^t m^2 + (2^{t+1} + 3)m + 2}.
\end{align}

\noindent Let $i_{t,M}$ be the unique solution to $2^t i \equiv 1 \pmod{M}$. Performing the shift $m \to m + i_{t,M}$ (respectively, $m \to m - i_{t,M}$) to the second (respectively, fourth) sum in (\ref{four}) followed by a routine simplification implies (\ref{mean}). The result now follows.
\end{proof}

For $\chi_t$ as in (\ref{genchi}), consider the series 

\begin{equation*}
H_t(x) = H_t(x,q) = \sum_{n \geq 0} \chi_{t}(n) q^{\frac{n^2 - (2^{t+1} - 3)^2}{3\cdot 2^{t+2}}} x^{\frac{n - (2^{t+1} - 3)}{2}}.
\end{equation*}

\begin{proposition}\label{step2} Let $h(t)=2^{t}-2$. We have

\begin{align}  \label{Htomulti}
H_t(x) & = (-1)^{h''(t)} q^{-h'(t)} x^{-h(t)} \sum_{n \geq 0} (x)_{n+1} x^{n m(t)} \nonumber \\
& \times \sum_{3 \sum_{\ell = 1}^{m(t) - 1} j_{\ell} \ell \equiv 1 \pmod{m(t)}} (-x)^{\sum_{\ell=1}^{m(t)-1} j_{\ell}} q^{\frac{-a(t) + \sum_{\ell=1}^{m(t) - 1} j_{\ell} \ell}{m(t)} + \sum_{\ell=1}^{m(t)-1} \binom{j_{\ell}}{2}} \nonumber \\
& \times \sum_{k=0}^{m(t)-1} x^k \prod_{\ell=1}^{m(t) - 1} \begin{bmatrix} n + I(\ell \leq k) \\ j_{\ell} \end{bmatrix}.
\end{align}
\end{proposition}

\begin{proof}
Let $f_t(x)$ denote the right-hand side of (\ref{Htomulti}). Then $f_t(x)$ satisfies the difference equation (see (3.3.11) in \cite{konan})

\begin{equation} \label{diff}
f_t(x) = 1 - q^2 x^3 - q^{2^t - 1} x^{2^t} + q^{3+2^t} x^{3+2^t} + q^{5\cdot 2^t - 3} x^{3\cdot 2^t} f_t(q^2 x).
\end{equation}

\noindent One can directly verify that $H_t(x)$ also satisfies (\ref{diff}) using (\ref{genchi}). The result now easily follows.
\end{proof}

Recall that the Kontsevich-Zagier series $F(q)$ satisfies the ``strange identity"

\begin{equation} \label{strid}
F(q) `` = " -\frac{1}{2} \sum_{n \geq 1} n \Bigl( \frac{12}{n} \Bigr) q^{\frac{n^2 - 1}{24}}
\end{equation}

\noindent where $``="$ means that the two sides agree to all orders at every root of unity (for further details, see Sections 2 and 5 in \cite{z1}) and $\bigl( \frac{12}{*} \bigr)$ is the quadratic character of conductor $12$. We now prove the following new strange identity for the Kontsevich-Zagier series for torus knots $T(3,2^t)$\footnote{Using the convention in footnote \ref{t1}, we can take $t=1$ in (\ref{genchi}) and (\ref{strangeFt}) to recover (\ref{strid}).}.

\begin{theorem} \label{newstr} For each root of unity $\zeta$, we have the asymptotic expansion

\begin{equation} \label{strangeFt}
\mathscr{F}_t(\zeta e^{-u}) \sim -\frac{1}{2} P_{(2^{t+1} -3)^2, 3\cdot 2^{t+2},\chi_t}^{(1)}(\zeta e^{-u})
\end{equation}

\noindent as $u \to 0^{+}$.
\end{theorem}

\begin{proof}[Proof of Theorem \ref{newstr}]
For ease of notation, let ${}^{'}$ denote the condition

\begin{equation*}
3 \sum_{\ell=1}^{m(t)-1} j_{\ell} \ell \equiv 1 \pmod{m(t)}
\end{equation*}

\noindent occurring in the second sum in (\ref{t32t}) on the $j_{\ell}$,

\begin{equation*}
v=v(j_1, \dotsc, j_{m(t)-1}) := \frac{-a(t) + \sum_{\ell=1}^{m(t) - 1} j_{\ell} \ell}{m(t)} + \sum_{\ell=1}^{m(t)-1} \binom{j_{\ell}}{2}
\end{equation*}

\noindent and $\overline{c}$ the reduction of an integer $c$ modulo $m(t)$. The following identity\footnote{The $t=2$ case of (\ref{Keyidgen}) gives an alternative (and corrected) version of Proposition 5 in \cite{hk2}.} implies (\ref{strangeFt}) upon setting $q=\zeta e^{-u}$, then letting $u \to 0^{+}$:

\begin{align} \label{Keyidgen}
\frac{1}{2} & \sum_{n \geq 0} n \chi_{t}(n) q^{\frac{n^2 - (2^{t+1} - 3)^2}{3\cdot 2^{t+2}}} - \frac{2^{t+1} - 3}{2} (q^{2^t - 1}, q^{2^t + 1}, q^{2^{t+1}}; q^{2^{t+1}})_{\infty} (q^2, q^{2^{t+2} - 2}; q^{2^{t+2}})_{\infty} \nonumber \\
& = (-1)^{h''(t)+1} q^{-h'(t)}  \sum_{n \geq 0} \Bigl[ (q)_n - (q)_{\infty} \Bigr] \nonumber \\
& \qquad \qquad \qquad \qquad \qquad \times \sideset{}{'} \sum_{j_1, \dotsc, j_{m(t)-1}} (-1)^{\sum_{{\ell}=1}^{m(t)-1} j_{\ell}} q^{v} \sum_{k=0}^{m(t)-1} \prod_{{\ell}=1}^{m(t)-1} \begin{bmatrix} n + I(\ell \leq k) \\ j_{\ell} \end{bmatrix} \nonumber \\
& + (-1)^{h''(t)+1} q^{-h'(t)} (q)_{\infty} \Biggl( \sum_{i=1}^{\infty} \frac{q^i}{1-q^i} \Biggr)  \sum_{n \geq 0} b_{n,t}(q) + (-1)^{h''(t)} q^{-h'(t)} (q)_{\infty} \sum_{n \geq 0} (n - h(t)) b_{n,t}(q) 
\end{align}

\noindent where

\begin{equation} \label{bnt}
b_{n,t}(q) = a_{n,t}(q) - a_{n-1,t}(q)
\end{equation}

\noindent and

\begin{equation} \label{ant}
a_{n,t}(q) = \sideset{}{'} \sum_{j_1, \dotsc, j_{m(t)-1}} (-1)^{\sum_{\ell =1}^{m(t)-1} j_{\ell}} q^{v} \prod_{\ell =1}^{m(t)-1} \begin{bmatrix} \frac{n - \sum_{\ell =1}^{m(t)-1} j_{\ell} - (\overline{n - \sum_{\ell =1}^{m(t)-1} j_{\ell}})}{m(t)} + I(\ell \leq \overline{n - \sum_{\ell =1}^{m(t)-1} j_{\ell}}) \\ j_{\ell} \end{bmatrix}.
\end{equation}

\noindent To prove (\ref{Keyidgen}), we begin by rewriting (\ref{Htomulti}) as

\begin{align} \label{rewrite}
H_t(x) & =  (-1)^{h''(t)} q^{-h'(t)} (1-x) \sum_{n \geq 0} \Bigl[ (qx)_n - (qx)_{\infty} \Bigr] x^{n m(t)} \nonumber \\
& \times  \sideset{}{'} \sum_{j_1, \dotsc, j_{m(t)-1}} (-1)^{\sum_{\ell =1}^{m(t)-1} j_{\ell}} x^{-h(t) + \sum_{\ell =1}^{m(t)-1} j_{\ell}} q^{v} \sum_{k=0}^{m(t)-1} x^k \prod_{\ell=1}^{m(t) - 1} \begin{bmatrix} n + I(\ell \leq k) \\ j_{\ell} \end{bmatrix} \nonumber \\
& + (qx)_{\infty} (-1)^{h''(t)} q^{-h'(t)} x^{-h(t)} (1-x) M_t(x,q)
\end{align}

\noindent where

\begin{equation} \label{Mdef}
M_t(x,q) := \sum_{n \geq 0} x^{n m(t)} \sideset{}{'} \sum_{j_1, \dotsc, j_{m(t)-1}} (-x)^{\sum_{\ell =1}^{m(t)-1} j_{\ell}} q^{v} \sum_{k=0}^{m(t)-1} x^k \prod_{\ell=1}^{m(t) - 1} \begin{bmatrix} n + I(\ell \leq k) \\ j_{\ell} \end{bmatrix}.
\end{equation}

\noindent We now claim that 

\begin{equation} \label{rewrite2}
(1-x) M_{t}(x,q) = \sum_{n \geq 0} b_{n,t}(q) x^n
\end{equation}

\noindent where $b_{n,t}(q)$ is given by (\ref{bnt}) and (\ref{ant}). To see (\ref{rewrite2}), we first write

\begin{equation} \label{rewrite3}
\sum_{\ell = 1}^{m(t) -1 } j_{\ell} = m(t) \Biggl \lfloor \frac{\sum_{\ell=1}^{m(t)-1} j_{\ell}}{m(t)} \Biggr \rfloor + \overline{\sum_{\ell=1}^{m(t)-1} j_{\ell}}.
\end{equation}

\noindent We now have that $M_t(x,q)$ equals

\begin{align*}
& \sideset{}{'} \sum_{j_1, \dotsc, j_{m(t)-1}} \sum_{k=0}^{m(t)-1} \sum_{n \geq 0} (-1)^{\sum_{\ell = 1}^{m(t)-1} j_{\ell}} x^{m(t) \Bigl (n + \Bigl \lfloor \frac{\sum_{\ell=1}^{m(t)-1} j_{\ell}}{m(t)} \Bigr \rfloor \Bigr) + k + \overline{\sum_{\ell=1}^{m(t)-1} j_{\ell}}} q^v \prod_{\ell=1}^{m(t) - 1} \begin{bmatrix} n + I(\ell \leq k) \\ j_{\ell} \end{bmatrix} \\
& \quad \quad \quad (\text{using (\ref{Mdef}) and (\ref{rewrite3})}) \\
& = \sideset{}{'} \sum_{j_1, \dotsc, j_{m(t)-1}} \sum_{k=0}^{m(t)-1} \sum_{n \geq 0} (-1)^{\sum_{\ell =1}^{m(t)-1} j_{\ell}} x^{m(t)n + k + \overline{\sum_{\ell=1}^{m(t)-1} j_{\ell}}} q^v \prod_{\ell=1}^{m(t) - 1} \begin{bmatrix} n - \Bigl \lfloor \frac{\sum_{\ell=1}^{m(t)-1} j_{\ell}}{m(t)} \Bigr \rfloor + I(\ell \leq k) \\ j_{\ell} \end{bmatrix} \\
& \quad \quad \quad \Biggl (\text{letting $n \to n - \Biggl \lfloor \frac{\sum_{\ell=1}^{m(t)-1} j_{\ell}}{m(t)} \Biggr \rfloor$} \Biggr) \\
& = \sum_{k=0}^{m(t)-1} \sideset{}{'} \sum_{j_1, \dotsc, j_{m(t)-1}} \sum_{n \equiv k +  \overline{\sum_{\ell=1}^{m(t)-1} j_{\ell}} \pmod{m(t)}} (-1)^{\sum_{\ell =1}^{m(t)-1} j_{\ell}} x^n q^v  \prod_{\ell=1}^{m(t) - 1} \begin{bmatrix} \frac{n - k  - \sum_{\ell=1}^{m(t)-1} j_{\ell}}{m(t)} + I(\ell \leq k) \\ j_{\ell} \end{bmatrix} \\
& \quad \quad \quad \Biggl (\text{letting $n \to \frac{n - k - \overline{\sum_{\ell=1}^{m(t)-1} j_{\ell}}}{m(t)}$ and using (\ref{rewrite3})} \Biggr) \\
& = \sum_{n \geq 0} \sideset{}{'} \sum_{j_1, \dotsc, j_{m(t)-1}} x^n (-1)^{\sum_{\ell =1}^{m(t)-1} j_{\ell}} q^{v} \prod_{\ell =1}^{m(t)-1} \begin{bmatrix} \frac{n - \sum_{\ell =1}^{m(t)-1} j_{\ell} - (\overline{n - \sum_{\ell =1}^{m(t)-1} j_{\ell}})}{m(t)} + I(\ell \leq \overline{n - \sum_{\ell =1}^{m(t)-1} j_{\ell}}) \\ j_{\ell} \end{bmatrix} \\
& \quad \quad \quad \Biggl (\text{combining the outer and innermost sum and replacing $k$ with $\overline{n - \sum_{\ell =1}^{m(t)-1} j_{\ell}}$} \Biggr) \\
& = \sum_{n \geq 0} a_{n,t}(q) x^n.
\end{align*}

\noindent Thus, we obtain (\ref{rewrite2}) via (\ref{bnt}). We now substitute (\ref{rewrite2}) into (\ref{rewrite}), then take the derivative of both sides of (\ref{rewrite}) with respect to $x$ and set $x=1$. This yields the right-hand side of (\ref{Keyidgen}). To obtain the left-hand side of (\ref{Keyidgen}), we additionally break up the resulting second sum into four sums, one for each congruence class modulo $3 \cdot 2^{t+1}$ in (\ref{genchi}). Namely,

\begin{align} \label{four2}
\sum_{n \geq 0} \chi_t(n) q^{\frac{n^2 - (2^{t+1} - 3)^2}{3\cdot 2^{t+2}}} & = \sum_{k \geq 0} \Bigl( q^{3\cdot 2^t k^2 + (2^{t+1} - 3)k} - q^{3\cdot 2^t k^2 + (2^{t+2} - 3)k + 2^t - 1} \Bigr) \nonumber \\
& + \sum_{k \geq 0} \Bigl( q^{3 \cdot 2^t k^2 + (2^{t+2} + 3)k + 2^t + 3} - q^{3\cdot 2^t k^2 + (2^{t+1} + 3)k + 2} \Bigr).
\end{align}
Let $k \to k-1$, then replace $k$ with $-k$ in the third and fourth sums of (\ref{four2}) and combine with the first and second sums in (\ref{four2}) to obtain

\begin{equation} \label{lstep}
\sum_{k \in \mathbb{Z}} \Bigl( q^{3\cdot 2^t k^2 + (2^{t+1} - 3)k} - q^{3\cdot2^t k^2 + (2^{t+2} - 3)k + 2^{t} - 1} \Bigr).
\end{equation}
We now apply Watson's quintiple product identity

\begin{equation*} \label{wqpi}
\sum_{k \in \mathbb{Z}} q^{\frac{k(3k-1)}{2}} x^{3k} (1-xq^k) = (q, x, qx^{-1})_{\infty} (qx^2, qx^{-2}; q^2)_{\infty}
\end{equation*}

\noindent with $q \to q^{2^{t+1}}$ and $x=q^{2^t -1}$ to (\ref{lstep}). This proves the result.
\end{proof}

\begin{remark} If we take $x=1$ in (\ref{rewrite}), apply Watson's quintiple product identity as above and use the fact that the sum in (\ref{rewrite2}) telescopes, then we have

\begin{align} \label{genslater}
(q)_{\infty} (-1)^{h''(t)} q^{-h'(t)} \sideset{}{'} \sum_{j_1, \dotsc, j_{m(t)-1}} & (-1)^{\sum_{\ell =1}^{m(t)-1} j_{\ell}} \frac{q^{v}}{(q)_{j_1} \cdots (q)_{j_{m(t)-1}}} \nonumber \\
& = (q^{2^t - 1}, q^{2^t + 1}, q^{2^{t+1}}; q^{2^{t+1}})_{\infty} (q^2, q^{2^{t+2} - 2}; q^{2^{t+2}})_{\infty}.
\end{align}
This recovers Slater's identity (see (86) in \cite{slater})
\begin{equation*}
(q)_{\infty} \sum_{n \geq 0} \frac{q^{2n(n+1)}}{(q)_{2n+1}} = (q^3, q^5, q^8; q^8)_{\infty} (q^2, q^{14}; q^{16})_{\infty}
\end{equation*} 
upon taking $t=2$ in (\ref{genslater}) and simplifying. It would be interesting to give an alternative proof of (\ref{genslater}), possibly using Bailey pairs or a combinatorial interpretation.
\end{remark}

\section{Proof of Theorem \ref{main}}

\begin{proof}[Proof of Theorem \ref{main}]
Let $p$ be a prime $\geq 5$ and $n \geq r$ be an integer. Consider the truncation $\mathscr{F}_{t}(q; pn-1)$ of (\ref{kztorus}) and its $p$-dissection

\begin{equation*}
\mathscr{F}_{t}(q; pn-1) = \sum_{i=0}^{p-1} q^i A_{\mathscr{F}_t, p}(pn-1, i, q^p).
\end{equation*}

\noindent We have

\begin{equation*}
\mathscr{F}_{t}(1-q; pn-1) = \sum_{i \in S_{t, \chi_t}(p)} (1-q)^i A_{\mathscr{F}_t, p}(pn-1, i, (1-q)^p) + \sum_{i \not \in S_{t, \chi_t}(p)} (1-q)^i A_{\mathscr{F}_t, p}(pn-1, i, (1-q)^p).
\end{equation*}

\noindent By Theorem \ref{strangetocong}, Proposition \ref{step1} and Theorem \ref{newstr}, we can write

\begin{equation*}
\mathscr{F}_{t}(1-q; pn-1) = \sum_{i \in S_{t, \chi_t}(p)} (1-q)^i A_{\mathscr{F}_t, p}(pn-1, i, (1-q)^p) + (1-(1-q)^p)^n \sum_{i \not \in S_{t, \chi_t}(p)} (1-q)^i g_i(q)
\end{equation*}

\noindent for $g_i(q) \in \mathbb{Z}[q]$. By Lemma 3.3 in \cite{straub},

\begin{equation*} 
(1-(1-q)^p)^n \equiv O(q^{pn - (p-1)(r-1)}) \pmod{p^r}
\end{equation*}

\noindent and so

\begin{equation*} \label{almost}
\mathscr{F}_{t}(1-q; pn-1) = \sum_{i \in S_{t, \chi_t}(p)} (1-q)^i A_{\mathscr{F}_t, p}(pn-1, i, (1-q)^p) + O(q^{pn - (p-1)(r-1)}) \pmod{p^r}.
\end{equation*}

\noindent Choosing $n$ large enough, it suffices to show that the coefficient of $q^{p^r m - j}$ in $(1-q)^i A_{\mathscr{F}_t, p}(pn-1, i, (1-q)^p)$ vanishes modulo $p^r$ for all $i \in S_{t, \chi_t}(p)$. Since $A_{\mathscr{F}_t, p}(pn-1, i, (1-q)^p) \in \mathbb{Z}[q]$, it suffices to show that 

\begin{equation} \label{cong}
\binom{i+lp}{p^r m - j} \equiv 0 \pmod{p^r}
\end{equation}

\noindent for natural numbers $l$ and $j \in \{1, 2, \dotsc, p - 1 - \text{max} \,\, S_{t, \chi_t}(p) \}$. The condition on $j$ implies that $j < p - i$ and so (\ref{cong}) follows from Lemma 3.4 in \cite{straub}. This proves the result.
\end{proof}

\section{Further comments}

There are several avenues for further study. First, one could investigate combinatorial descriptions and asymptotic properties for the numbers $\xi_t(n)$. Second, the colored Jones polynomial of a knot $K$ satisfies a \emph{cyclotomic expansion} of the form
\begin{equation*} 
J_N(K;q) = \sum_{n \geq 0} (q^{1+N})_n(q^{1-N})_n C_n(K;q),
\end{equation*}
where the \emph{cyclotomic coefficients} $C_n(K;q)$ are Laurent polynomials independent of $N$ \cite{habiro2}. It is highly desirable to find the cyclotomic expansion for the torus knots $T(3,2^t)$, $t \geq 2$. For $t=1$, the cyclotomic expansion has been found by Masbaum \cite{masbaum}. However, it is unclear if his techniques are sufficient when $t \geq 2$. They require finding a link whose components are unknots, from which $T(3,2)$ can be recovered by introducing twists into a single region of one of the components. Employing this process for $t \geq 2$ appears to require extending it to allow for multiple twist regions. Third, in relation to Theorem \ref{newstr}, a strange identity for the Kontsevich-Zagier series associated to the torus knots $T(2, 2t+1)$ for $t \geq 1$ has also been computed (see (15) in \cite{hikami2}). Do strange-type identities exist for similar $q$-series associated to satellite or hyperbolic knots? Fourth, as the Kontsevich-Zagier series $F(q)$ is a foundational example of a quantum modular form \cite{zagier}, it is natural to wonder if the same is true for $\mathscr{F}_t(q)$. This is the subject of the paper \cite{ano}. Finally, we are pleased to report that the techniques in this paper have recently been extended in \cite{an} to prove prime power congruences for additional families of generalized Fishburn numbers.

\section*{Acknowledgements}

The second author was partially funded by the Natural Sciences and Engineering Research Council of Canada. The fourth author would like to thank the Max-Planck-Institut f\"ur Mathematik for their support during the initial stages of this project, the Ireland Canada University Foundation for the James M. Flaherty Visiting Professorship award, McMaster University for their hospitality during his stay from May 17 to August 9, 2019 and Armin Straub and Byungchan Kim for their computational help. The first, third, sixth and seventh authors graciously thank the Fields Institute for the opportunity to take part in their 2019 Undergraduate Summer Research Program. Finally, the authors thank the referee for their careful reading and helpful suggestions.

\end{document}